\documentclass[reqno,  twoside, a4paper, 11pt]{amsart}
\usepackage[english]{babel}
\usepackage{amsmath}
\usepackage[colorlinks=true,draft=false]{hyperref}
\usepackage{amsfonts}
\usepackage{thmtools}
\usepackage{amssymb}
\usepackage[pagewise,mathlines]{lineno}
\usepackage{amsthm}
\usepackage{fixme}
\usepackage{tikz}
\usepackage{tikz-cd}
\usepackage{stmaryrd}
\usepackage{MnSymbol}
\usepackage{microtype}
\usepackage{mathrsfs}
\usepackage[all]{xy}
\usepackage{rotating}
\usepackage{enumitem}

\usepackage{cleveref}

\newcount\hh
\newcount\mm
\mm=\time
\hh=\time
\divide\hh by 60
\divide\mm by 60
\multiply\mm by 60
\mm=-\mm
\advance\mm by \time
\def\hhmm{\number\hh:\ifnum\mm<10{}0\fi\number\mm}

\DeclareMathOperator{\pr}{pr}

\DeclareMathOperator{\rs}{rs}

\DeclareMathOperator{\Pol}{Pol}

\DeclareMathOperator{\trdeg}{trdeg}

\DeclareMathOperator{\codim}{codim}

\DeclareMathOperator{\Spec}{Spec}

\DeclareMathOperator{\Sym}{Sym}

\DeclareMathOperator{\Strat}{Strat}

\DeclareMathOperator{\red}{red}

\DeclareMathOperator{\Frac}{Frac}

\DeclareMathOperator{\reg}{reg}

\newcommand{\Z}{\mathbb{Z}}

\newcommand{\N}{\mathbb{N}}

\newcommand{\A}{\mathbb{A}}
\newcommand{\llparen}{(\!(}
\newcommand{\rrparen}{)\!)}

\renewcommand{\P}{\mathbb{P}}

\renewcommand{\subset}{\subseteq}
\renewcommand{\supset}{\supseteq}

\renewcommand{\phi}{\varphi}
\declaretheoremstyle[spaceabove=15pt,spacebelow=15pt]{defstyle}
\declaretheoremstyle[ spaceabove=15pt, spacebelow=15pt, bodyfont=\itshape]{thmstyle}
\declaretheorem[numberwithin=section, style=thmstyle]{Theorem}
\declaretheorem[sibling=Theorem, style=thmstyle]{Proposition}
\declaretheorem[sibling=Theorem,  style=thmstyle]{Lemma}
\declaretheorem[sibling=Theorem, style=thmstyle]{Corollary}

\declaretheorem[sibling=Theorem, style=defstyle]{Definition}
\declaretheorem[sibling=Theorem, style=defstyle]{Remark}

\setenumerate{label=\emph{(\alph*)}, ref=(\alph*)}

\title[Pullback and restriction to curves]{Pullback of regular singular stratified bundles and restriction to
curves}
\author{Lars Kindler}
\address{Mathematisches Institut\\ Freie Universit\"at Berlin\\ Arnimallee 3\\ 14195
Berlin, Germany}
\email{kindler@math.fu-berlin.de}
\urladdr{http://mi.fu-berlin.de/\textasciitilde kindler/}

\thanks{This work was supported ERC Advanced Grant 226257.}

\begin{document}
\maketitle
\begin{abstract}
	A stratified bundle on a smooth variety $X$ is a vector bundle which is a
	$\mathscr{D}_{X}$-module.
	We show that regular singularity of stratified bundles on smooth
	varieties in positive characteristic is preserved by
	pullback and that regular singularity can be checked on curves, if the ground
	field is large enough.
\end{abstract}
\section{Introduction}
\def\sectionautorefname{Section}

If $X$ is a smooth complex variety, then it is proved in
\cite{Deligne/RegularSingular} that a vector bundle with flat connection
$(E,\nabla)$ on $X$ is regular singular if and only $\phi^*(E,\nabla)$ is regular
singular for all maps $\phi:C\rightarrow X$ with $C$ a smooth complex curve.

In this short note we analyze an analogous statement for vector bundles with
$\mathscr{D}_{X/k}$-action on
smooth $k$-varieties, where $k$ is an algebraically closed field of positive
characteristic $p>0$ and $\mathscr{D}_{X/k}$ the sheaf of differential
operators of $X$ relative to $k$. Vector bundles with such an action are called stratified
bundles, see \cite{Gieseker/FlatBundles}. A notion of regular singularity for stratified bundles was
defined and studied in \emph{loc.~cit.} under the assumption of the
existence of a good compactification, and in \cite{Kindler/FiniteBundles} in
general. We recall this definition in Section \ref{sec:regsing}.

The first result of this article is:

\begin{Theorem}\label{thm:main:curves}
	Let $k$ be an  \emph{uncountable} algebraically closed field of characteristic $p>0$, $X$ a smooth, separated,
	finite type $k$-scheme and $E$ a stratified bundle
	on $X$. Then $E$ is regular singular if and only if $\phi^*E$ is
	regular singular for every $k$-morphism $\phi:C\rightarrow X$ with
	$C$ smooth $k$-curve.
\end{Theorem}

In
\cite[Sec.~8]{Kindler/FiniteBundles} it is proved that \autoref{thm:main:curves} holds
without the uncountability condition for stratified bundles with finite monodromy. This relies on
work of Kerz, Schmidt and Wiesend, \cite{Kerz/tameness}. For stratified
bundles with arbitrary monodromy, the author does not know at present whether the uncountability
condition of $k$ in \autoref{thm:main:curves} is necessary or not. For this reason we
have to content ourselves with the following general criterion for regular
singularity, which easily follows from \autoref{thm:main:curves}.
\begin{Corollary}\label{cor:arbitraryField}
	If $X$ is a smooth, separated, finite type $k$-scheme, and $E$ a 
	stratified bundle on $X$, then $E$ is regular singular if and only if
	for every algebraic closure $k'$ of a finitely generated extension of
	$k$, and every $k'$-morphism $\phi:C\rightarrow X_{k'}$ with $C$ a
	smooth $k'$-curve, the stratified bundle $\phi^*(E\otimes k')$ on $C$ is
	regular singular.
\end{Corollary}

In the course of the proof we establish a general result on pullbacks, which
is of independent interest.
\begin{Theorem}\label{thm:main:pullback}
	Let $k$ be an algebraically closed field of characteristic $p>0$ and
	$f:Y\rightarrow X$ a morphism of smooth, separated, finite type
	$k$-schemes. If $E$ is a regular singular stratified bundle on $X$,
	then $f^*E$ is a regular singular stratified bundle on $Y$.

	In other words,
if $\Strat(X)$  denotes
the category of stratified bundles on $X$ and $\Strat^{\rs}(X)$ its full
subcategory with objects the regular singular stratified bundles,
 then the pullback functor
$f^*:\Strat(X)\rightarrow \Strat(Y)$ restricts to a functor
$f^*:\Strat^{\rs}(X)\rightarrow \Strat^{\rs}(Y)$. 
\end{Theorem}

The difficulty in proving this theorem is the unavailability of resolution of
singularities.
Our proof relies on a desingularization result kindly
communicated to H.~Esnault and the author by O.~Gabber.
In \cite{Kindler/thesis}, \autoref{thm:main:pullback}  was only shown in the case that $f$ is
dominant.\\

We conclude the introduction with a brief outline of the article. In \autoref{sec:regsing} we recall
the definition of regular singularity of a stratified bundle via good partial
compactifications. In \autoref{sec:pullback} we prove
\autoref{thm:main:pullback}, and in \autoref{sec:curves} we establish
\autoref{thm:main:curves}.

\vspace{.5cm}
\section*{Acknowledgements} The author wishes to  thank H.~Esnault, M.~Morrow and
K.~R\"ulling for enlightening
discussions, and O.~Gabber for communicating \autoref{lemma:Gabber} to us.

\section{Regular singular stratified bundles}\label{sec:regsing}
Let $k$ denote an algebraically closed field and $X$ a smooth
$k$-variety, i.e.~a smooth, separated, finite type $k$-scheme. If $k$ has
characteristic $0$, then giving a flat connection $\nabla$ on a vector bundle
$E$ on $X$ is equivalent to giving $E$ a left-$\mathscr{D}_{X/k}$-action which
is compatible with its $\mathcal{O}_X$-structure.
Here $\mathscr{D}_{X/k}$ is the sheaf of rings of differential operators of
$X$ relative to $k$ (\cite[\S16]{EGA4}). If $k$ has positive characteristic, the sheaf of rings
$\mathscr{D}_{X/k}$ is still defined, but a flat connection does not
necessarily give rise to a $\mathscr{D}_{X/k}$-module structure on $E$.

\begin{Definition}\label{defn:stratifiedBundle}
	If $X$ is a smooth, separated, finite type $k$-scheme, then a
	\emph{stratified bundle $E$ on $X$} is a 
	left-$\mathscr{D}_{X/k}$-module $E$ which is coherent with respect to the
	induced $\mathcal{O}_X$-structure. The category of such objects is
	denoted by $\Strat(X)$.
\end{Definition}

The usage of the word ``bundle'' is justified as a stratified bundle is automatically
locally free as an $\mathcal{O}_X$-module. See, e.g.,
\cite[2.17]{BerthelotOgus/Crystalline}.  The analogous statement for flat
connections is not true in positive characteristic. The notion of a stratification
goes back to \cite{Grothendieck/Crystals} and a vector bundle on a
\emph{smooth}
$k$-scheme equipped  with a
stratification  relative to $k$ in the sense of \emph{loc.~cit.}~is a stratified bundle in the sense of
\autoref{defn:stratifiedBundle}, and vice versa.

To define regular singularity of a stratified bundle, we introduce some
notation.

\begin{Definition}\begin{itemize}
	\item If $\overline{X}$ is a smooth, separated, finite type $k$-scheme
		and $X\subset \overline{X}$ an open subscheme such that
		$\overline{X}\setminus X$ is the support of a strict normal crossings
		divisor, then the pair $(X,\overline{X})$ is called \emph{good partial
		compactification of $X$}. If in addition $\overline{X}$ is proper over
		$k$, then $(X,\overline{X})$ is called \emph{good compactification of
		$X$}.
	\item If $(Y,\overline{Y})$ and $(X,\overline{X})$ are two good
		partial compactifications, then a \emph{morphism of good
		partial compactifications} is a morphism
		$\bar{f}:\overline{Y}\rightarrow \overline{X}$ such that
		$\bar{f}(Y)\subset X$.
	\item If $(Y,\overline{Y})$, $(X,\overline{X})$ are good partial
		compactifications and if $f:Y\rightarrow X$ is a morphism, then
		we say that \emph{$f$ extends to a morphism of good partial
		compactifications} if there exits a morphism of good partial
		compactifications $\bar{f}:(Y,\overline{Y})\rightarrow
		(X,\overline{X})$, such that $f=\bar{f}|_Y$.\end{itemize}
\end{Definition}

Let $X$ be a smooth, separated, finite type $k$-scheme. If $k$ has
characteristic $0$,  then there
exists a good compactification $(X,\overline{X})$ according to 
Hironaka's theorem on resolution of singularities. By definition, a vector bundle with flat connection $(E,\nabla)$ on $X$ is 
regular singular, if for some (and equivalently for any) good compactification $(X,\overline{X})$, there
exists a torsion free, coherent $\mathcal{O}_{\overline{X}}$-module
$\overline{E}$ extending $E$ and a logarithmic connection
$\overline{\nabla}:\overline{E}\rightarrow
\overline{E}\otimes_{\mathcal{O}_{\overline{X}}}
\Omega^1_{\overline{X}/k}(\log \overline{X}\setminus X)$ extending $\nabla$.

If $k$ has positive characteristic then it is unknown whether every smooth
$X$ admits a good compactification. We work with \emph{all} good
\emph{partial} compactifications instead.

\begin{Definition}\label{def:gpc}
		Let $k$ be an algebraically closed field of positive
		characteristic and $X$ a smooth, separated, finite type
		$k$-scheme.
		\begin{enumerate}[label={{(\alph*)}}]
			\item If $E$ is a stratified bundle on $X$ and $(X,\overline{X})$ a good
				partial compactification, then $E$ is called
		\emph{$(X,\overline{X})$-regular singular} if there exists a
		$\mathscr{D}_{\overline{X}/k}(\log \overline{X}\setminus X)$-module
		$\overline{E}$, which is coherent and torsion free as an
		$\mathcal{O}_{\overline{X}}$-module, such that $E\cong
		\overline{E}|_X$ as stratified bundles. Here
		$\mathscr{D}_{\overline{X}/k}(\log \overline{X}\setminus X)$
		is the sheaf of differential operators with logarithmic poles
		along the  underlying normal crossings divisor $\overline{X}\setminus X$, as defined in
		\cite[Sec.~3]{Gieseker/FlatBundles}; see also
		\autoref{rem:logDiffOps}.
	\item A stratified bundle $E$ is called \emph{regular singular} if $E$ is
		$(X,\overline{X})$-regular singular for all good partial
		compactifications $(X,\overline{X})$. We write
		$\Strat^{\rs}(X)$ for the full subcategory of $\Strat(X)$ with
		objects the regular singular stratified bundles.
	\end{enumerate}
\end{Definition}
\begin{Remark}\label{rem:logDiffOps}
	A good partial compactification $(X,\overline{X})$ gives rise to a logarithmic
	structure on $\overline{X}$ (\cite{Kato/LogSchemes}), and the associated log-scheme is a log-scheme
over $\Spec k$ equipped with its trivial log-structure.  Associated to this
morphism of log-schemes is a sheaf of logarithmic differential operators
$\mathscr{D}_{(X,\overline{X})/k}$, which agrees with the sheaf of rings
$\mathscr{D}_{\overline{X}/k}(\log \overline{X}\setminus X)$ from
\cite[Sec.~3]{Gieseker/FlatBundles}. In local coordinates, if $x\in
\overline{X}$ is a closed point and $x_1,\ldots, x_n$ \'etale coordinates around $x$
such that in a neighborhood of $x$ the normal crossings divisor $\overline{X}\setminus X$ is defined by
$x_1\cdot\ldots\cdot x_r=0$ for some $r\leq n$, then
$\mathscr{D}_{\overline{X}/k}(\log \overline{X}\setminus X)$ is spanned by
the operators 
\[x_1^s\partial_{x_1}^{(s)},\ldots,
	x_r^s\partial_{x_r}^{(s)},\partial_{x_{r+1}}^{(s)},\ldots,
	\partial_{x_n}^{(s)}, s\in\Z_{\geq 0}\]
	where  $\partial_{x_i}^{(s)}$ is the differential operator such that
	\[\partial_{x_i}^{(s)}(x_j^t)=\begin{cases}0 & i\neq j\\
			\binom{t}{s}x_j^{t-s}& i=j.\end{cases}\]
We refer to \cite[Sec.~3]{Kindler/FiniteBundles} for more details.
\end{Remark}

The notion of regular singularity for stratified bundles is studied in
\cite[Sec.~3]{Gieseker/FlatBundles} for smooth varieties $X$ which admit a good
compactification and in 
\cite{Kindler/FiniteBundles} in general.

We conclude this section by recalling the following fact about regular singularity,
which we will use repeatedly in the sequel.
\begin{Proposition}\label{prop:genericity}
	Let $E$ be a stratified bundle on a smooth, separated, finite type
	$k$-scheme $X$.
	\begin{enumerate}[label={\emph{(\alph*)}}, ref={(\alph*)}]
		\item\label{item:prop:genericity1} If $(X,\overline{X})$ is a good partial compactification
			then $E$ is $(X,\overline{X})$-regular singular if and
			only if there exists an open subset
			$\overline{U}\subset \overline{X}$ with
			\[\codim_{\overline{X}}\left(\overline{X}\setminus
				\left(X\cup \overline{U}\right)\right)\geq
			2,\]
			such that $E|_{\overline{U}\cap X}$ is $(X\cap
			\overline{U},\overline{U})$-regular singular.
		\item\label{item:prop:genericity2} If there exists a dense open subset $U\subset X$ such
	that $E|_U$ is regular singular, then $E$ is regular singular.
	\end{enumerate}
\end{Proposition}
\begin{proof}
	\begin{enumerate}[label={{(\alph*)}}]
	\item This is \cite[Prop.~4.3]{Kindler/FiniteBundles}.
		\item Assume that $E|_U$ is regular singular. It follows from
			the first part of this proposition that all we have
			to show is that for any good partial compactification
			$(X,\overline{X})$ there exists an open subset
			$\overline{U}\subset \overline{X}$ containing all
			generic points of $\overline{X}\setminus X$, such that
			$(U,\overline{U})$ is a good partial compactification.
			Write $\eta_1,\ldots, \eta_d\in \overline{X}$ for the
			codimension $1$ points
			not contained in $X$. Let $\overline{U}'_i$ be an open
			neighborhood of $\eta_i$ and $Z_i$ the closure of
			$(\overline{U}_i'\cap X)\setminus U$ in
			$\overline{X}$. Defining 
			$\overline{U}_i:=(U\cup \overline{U}'_i)\setminus
			Z_i$, the open subset
			$\overline{U}:=\bigcup_{i=1}^d \overline{U}_i\subset
			\overline{X}$  does the job (note that $Z_i\cap
			U=\emptyset$).
	\end{enumerate}
\end{proof}

\begin{Remark}
Once we have proved \autoref{thm:main:pullback}, we will also know that if $E$ is a regular singular
stratified bundle on $X$ then $E$ is regular singular when restricted to any open
subset.
\end{Remark}

%
%
%
%
%

\section{Pullback of regular singular stratified bundles}\label{sec:pullback}
In this section we construct the pullback functor
for regular singular stratified bundles.  We first recall a basic fact, which
is obvious from the perspective of log-schemes.

\begin{Proposition}[{\cite[Prop.~4.4]{Kindler/FiniteBundles}}]\label{prop:logPullback}
	Let $k$ be an algebraically closed field of positive characteristic
	and $\bar{f}:(Y,\overline{Y})\rightarrow (X,\overline{X})$ a morphism
	of good partial compactifications over $k$ (\autoref{def:gpc}). Write
	$f:=\bar{f}|_Y:Y\rightarrow X$. If $E$
	is an $(X,\overline{X})$-regular singular stratified bundle on $X$,
	then $f^*E$ is $(Y,\overline{Y})$-regular singular.
\end{Proposition}

Next, we show that in order to prove \autoref{thm:main:pullback}, it suffices to study dominant morphisms and closed
immersions separately.

\begin{Proposition}\label{prop:closedOrDominant}
	\autoref{thm:main:pullback} is true, if and only if it is true for all
	closed immersions and all dominant morphisms. \end{Proposition}
\begin{proof}
	Without loss of generality we may assume that $X$ and $Y$ are
	connected.
	Let $f:Y\rightarrow X$ be as in \autoref{thm:main:pullback}, and let
	$i:Z\hookrightarrow X$ be the closed immersion given by
	the scheme theoretic image of $f$. Since $Y$ is reduced, so is $Z$. We
	factor $f$ as $Y\xrightarrow{g} Z\xrightarrow{i} X$. Note that $g$ is
	dominant. Since $Z$ is reduced, there exists an open subscheme
	$U\subset X$, such that $U\cap Z$ is regular. 	Define
	$V:=f^{-1}(U)$, and 
	consider the sequence of maps
	\[V\xrightarrow{g|_V}Z\cap U \xrightarrow{i|_{Z\cap U}}U\xrightarrow{j} 
		X\]
	where $j:U\hookrightarrow X$ is the open immersion.
	The two outer maps are dominant, the middle map is a closed
	immersion and all four varieties are regular. We apply the assumption
	of this proposition from right to left. 

	Let $E$ be a regular singular stratified bundle on $X$. By assumption
	$E|_U$ is regular singular on $U$, then $(i|_{Z\cap U})^*E|_U$ is
	regular singular on $Z\cap U$ and finally $g|_V^*(i_{Z\cap
	U})^*E|_U=f|_V^* E|_U$ is regular singular on $V$. According to
	\autoref{prop:genericity} this means that $f^*E$ is regular singular.
\end{proof}

From this proposition together with \autoref{prop:logPullback} we see directly that to prove
\autoref{thm:main:pullback}, it suffices to prove the following statement.
\begin{Proposition}\label{prop:pullbackGPC}
	Let $k$ be an algebraically closed field and $f:Y\rightarrow X$ a
	morphism of smooth, separated, finite type $k$-schemes.
	\emph{Assume that $f$ is either dominant or a closed immersion.}
	
	If
	$(Y,\overline{Y})$ is a good partial compactification, then there
	exist
	\begin{itemize}
		\item an open subset $\overline{V}\subset \overline{Y}$ containing
	all generic points of $\overline{Y}\setminus Y$, and
\item a good partial 
	compactification $(X,\overline{X})$, 
	such that $f$ induces a morphism
	of good partial compactifications
	\[\bar{f}:(\overline{V}\cap Y, \overline{V})\rightarrow
		(X,\overline{X}).\]
\end{itemize}
\end{Proposition}

The remainder of this section is devoted to the proof of
\autoref{prop:pullbackGPC}. We first treat the dominant case
(\autoref{prop:pullback-dominant}), then the case of a
closed immersion (\autoref{lemma:pull-back-closed-immersion}).

\begin{Lemma}
\label{prop:pullback-dominant}

	\autoref{prop:pullbackGPC} is true for dominant
	morphisms $f:Y\rightarrow X$.
\end{Lemma}
\begin{proof}
	This is essentially {\cite[Ex.~8.3.16]{Liu/2002}}. Without loss of
	generality, we may assume $X$ and $Y$ to be irreducible. If
	$(Y,\overline{Y})$ is a good partial compactification, we may assume
	that $\overline{Y}\setminus Y$ is a smooth divisor, say with generic
	point $\eta$. If $X'$ is a normal compactification of $X$, then after
	removing a closed subset of codimension $\geq 2$ from
	$\overline{Y}$, $f$ extends to a morphism $f':\overline{Y}\rightarrow
	X'$. If $f'(\eta)\in X$ there is nothing to do; we may take
	$\overline{X}=X$. Otherwise, {\cite[Ex.~8.3.16]{Liu/2002}} tells us
	that there is a blow-up $X''\rightarrow X'$ of $X'$ in
	$\overline{\{f'(\eta)\}}$ such that  $f'$ extends to a map $f'':\overline{Y}\rightarrow X''$
	such that $f''(\eta)$ is a normal codimension $1$ point of $X''$. We
	define $\overline{X}$ to be a suitable neighborhood of $f''(\eta)$ to
	finish the proof.
\end{proof}

Let $k$ be an algebraically closed field and $X$ a normal, irreducible,
separated, finite type $k$-scheme. We write $k(X)$ for the function field of
$X$. We recall a few basic definitions:
\begin{Definition}Let $v$ be a discrete valuation on $k(X)$.
	\begin{itemize}
		\item We write $\mathcal{O}_v\subset
			k(X)$ for its valuation ring, $\mathfrak{m}_v$ for the
			maximal ideal of $\mathcal{O}_v$ and $k(v)$ for its residue field.
		\item If $X'$ is a model of $k(X)$, then a point $x\in
			X'$ is called \emph{center of $v$}, if
			$\mathcal{O}_{X',x}\subset \mathcal{O}_v\subset k(X)$,
			and $\mathfrak{m}_v\cap
			\mathcal{O}_{X',x}=\mathfrak{m}_x$. 
		\item $v$ is called \emph{geometric} if there
			exists a model $X'$ of $k(X)$ such that $v$ has a
			center $\xi\in X'$ which is a normal codimension $1$ point.
			In this case $\mathcal{O}_{v}=\mathcal{O}_{X',\xi}$.
	\end{itemize}
\end{Definition}
\begin{Remark}Recall that if
			$X'$ is separated over $k$, then $v$ has at most one
			center on $X'$ and if $X'$ is proper, then $v$ has
			precisely one center on $X'$.
		\end{Remark}

\begin{Proposition}[{\cite[Ch.~8,
	Thm.~3.26]{Liu/2002}, \cite[Ch.8~, Ex.~3.14]{Liu/2002}}]\label{prop:abhyankar-inequality}\label{prop:blow-ups}
	Let $X$ be a normal, irreducible, separated, finite type $k$-scheme
	and
	$v$ a discrete valuation on $k(X)$. 
	\begin{enumerate}[label={\emph{(\alph*)}}]
		\item We have the inequality
	\begin{equation}\label{abhyankar-inequality}\trdeg_{k}k(v)\leq \dim X
		-1\end{equation}
	where $k(v)$ is the residue field of $\mathcal{O}_v$.
\item\label{item:geometricValuationsB} The discrete valuation $v$ is geometric if and only if equality holds
	in \eqref{abhyankar-inequality}.
\item Let  
	$X'$ be a
	normal, proper compactification of $X$ and $x_0$ the center of $v$ on
	$X'$. If $k(v)/k(x_0)$ is finitely generated, we can make
	\ref{item:geometricValuationsB} more precise:
	
	Define $X'_0:=X'$.
	Inductively define \[\phi_n:X'_n\rightarrow X'_{n-1}\]
	as the blow-up of $X'_{n-1}$ in the reduced closed subscheme defined
	by $\overline{\{x_{n-1}\}}$, where $x_{n-1}$ is the center of $v$ on $X'_{n-1}$.
	If \eqref{abhyankar-inequality} is an equality for $v$, then for large $n$,
	$\phi_n$ is an isomorphism, i.e.~for large $n$,
	$x_n$ is a codimension $1$ point of $X'_n$.
	\end{enumerate}
\end{Proposition}

Now we prove \autoref{prop:pullbackGPC} in the case where $f$ is a closed
immersion.

\begin{Lemma}\label{cor:embedded-blow-ups}
	Let $i:Y\hookrightarrow X$ be a closed immersion of normal,
	irreducible, separated, finite type $k$-schemes, and let $v$ be a
	discrete geometric valuation of $k(Y)$ with center $y$ on $Y$ such
	that $k(v)/k(y)$ is finitely generated. Write
	$X_0:=X$, $Y_0:=Y$, $y_0:=y$, and inductively define \[X_n\rightarrow
		X_{n-1}\] as the blow-up of $X_{n-1}$ in
		$\overline{\{y_{n-1}\}}$, $Y_{n}$ as the proper transform of
		$Y_{n-1}$, and $y_n$ as the center of $v$ on $Y_n$. Then for
		large $n$, the center $y_n$ of $v$ has codimension $1$ in $Y_n$.

\end{Lemma}
\begin{proof}
	This follows directly from \autoref{prop:blow-ups}, using that the
	induced map $Y_n\rightarrow Y_{n-1}$ between the proper transforms is naturally isomorphic to
	the blow-up of $Y_{n-1}$ in $\overline{\{y_{n-1}\}}$.
\end{proof}
\begin{Proposition}\label{lemma:pull-back-closed-immersion}\autoref{prop:pullbackGPC} is true for closed
	immersions.
\end{Proposition}
\begin{proof}
	Let $i:Y\hookrightarrow X$ be a closed immersion of smooth, connected,
	separated, finite type $k$-schemes.  Let $(Y,\overline{Y})$ be a good
	partial compactification. Without loss of generatlity we may assume
	that $\overline{Y}\setminus Y$ is irreducible, and hence (the support
	of) a smooth
	divisor. To prove the lemma, we may replace $\overline{Y}$ by open
	neighborhoods $\overline{V}$ of the generic point $\eta$ of
	$\overline{Y}\setminus Y$ and $Y$ by $\overline{V}\cap Y$.

	Let $X'$ be a normal, proper $k$-scheme containing $X$ as a dense open
	subscheme.  After possibly removing a closed
	subset of codimension $\geq 2$ from $\overline{Y}$ we may assume that
	$i$ extends to a morphism $i':\overline{Y}\rightarrow X'$. Note that
	$i'(\eta)\in X'\setminus X$: otherwise we would have $i'(\eta)\in
	i'(Y)$, as $i'(\overline{Y})$ is irreducible, and then the valuation
	on $k(Y)$ associated with $\eta$  would have two centers,
	which is impossible as $Y$ is separated over $k$.

	By \autoref{cor:embedded-blow-ups} there exists a modification $X''\rightarrow
	X'$, which is an isomorphism over $X$, such that (perhaps after again removing a closed subset of codimension
	$\geq 2$ from $\overline{Y}$) $i$ extends to a map
	$i'':\overline{Y}\rightarrow X''$ such that $i''(\eta)$ is
	a codimension $1$ point of the closure of  $i''(\overline{Y})$ in
	$X''$. Thus, replacing $X'$ by $X''$ we may assume that
	$i:Y\hookrightarrow X$ extends to a closed immersion
	$i':\overline{Y}\hookrightarrow X'$.
	

	It remains to show that we can replace $X'$
	by a chain of blow-ups with centers over $X'\setminus
	X$, and $\overline{Y}$ with its proper transform, such that $i'(\eta)$ is a regular
	point of $X'$ and a regular point of $X'\setminus X$.

	For this we use a desingularization result kindly
	communicated to us by Ofer Gabber.
	\begin{Lemma}[Gabber]\label{lemma:Gabber}
		\begin{enumerate}[label={\emph{(\alph*)}}, ref={(\alph*)}]	
			\item \label{gabber:a} Let $\mathcal{O}$ be a noetherian local integral domain and
		$\mathfrak{p}\subset\mathcal{O}$ a prime ideal such that $\mathcal{O}/\mathfrak{p}$ is of
		dimension $1$ and such that the normalization of
		$\mathcal{O}/\mathfrak{p}$ is finite over $\mathcal{O}/\mathfrak{p}$.
		Write $X:=X_0:=\Spec \mathcal{O}$,  $C:=C_0:=\Spec
		\mathcal{O}/\mathfrak{p}$. For $n\geq 0$ let $X_{n+1}$ be the blow-up of $X_n$ at the closed
		points of $C_n$, and $C_{n+1}$ the proper transform of $C_n$ in
		$X_{n+1}$.  For large $n$, $C_n$ is regular, and at
		every closed point $\xi$ of $C_n$ we have that if
		\[\mathfrak{p}_n:=\ker\left(\mathcal{O}_{X_n,\xi}\rightarrow
			\mathcal{O}_{C_n,\xi}\right),\]
			then for every $m\in \N$ the $\mathcal{O}_{C_n,\xi}$-module
			$\mathfrak{p}_n^m/\mathfrak{p}_n^{m+1}$ is torsion-free. 

		\item \label{gabber:b}If $\mathcal{O}_{\mathfrak{p}}$ is
			regular, then for large $n$
			so is $\mathcal{O}_{X_n,\xi}$.
	\end{enumerate}
	\end{Lemma}
	\begin{proof}
			The normalization of $\mathcal{O}/\mathfrak{p}$ is
			finite, and it is obtained by blowing up singular
			points repeatedly. Thus we may assume that
			$\mathcal{O}/\mathfrak{p}$ is a discrete valuation
			ring.

			After this reduction, the map
			$C_{n+1}\rightarrow C_n$ induced by the blow-up
			$X_{n+1}\rightarrow X_n$ is an isomorphism. In
			particular, $C_n=C_{n-1}=\ldots=C_0=\Spec
			\mathcal{O}/\mathfrak{p}$.

			Write $\mathcal{O}'$ for the local ring of
			$X_{1}$ in the closed point of $C_{1}$, and $\pi\in
			\mathcal{O}$ for a lift of a uniformizer of the
			discrete valuation ring $\mathcal{O}/\mathfrak{p}$. To
			ease notation we will also write $\pi$ for its image
			in $\mathcal{O}/\mathfrak{p}$.
			Then $\mathcal{O}'$ is the localization of the ring
			$\sum_{i\geq 0}\pi^{-i}\mathfrak{p}\subset
			\Frac(\mathcal{O})$ at a suitable maximal ideal.
			Moreover, $\mathfrak{p}_1$ is the localization of the
			ideal generated by $\pi^{-1}\mathfrak{p}$. From this
			it is not difficult to see that we get a surjective
			$\mathcal{O}/\mathfrak{p}$-linear morphism
			\[\phi:\pi^{-m}\mathcal{O}/\mathfrak{p}\otimes_{\mathcal{O}/\mathfrak{p}}
				\mathfrak{p}^m/\mathfrak{p}^{m+1}\twoheadrightarrow
				\mathfrak{p}_1^m/\mathfrak{p}_1^{m+1},\]
			defined by $\pi^{-m}\otimes x \mapsto
			\pi^{-m}x$.
			
			The $\mathcal{O}/\mathfrak{p}$-module $\ker(\phi)$ is
			torsion. Indeed, if $x\in \mathfrak{p}^m$ is an
			element such that $\pi^{-m}x\in
			\mathfrak{p}_1^{m+1}$, then $\pi^{-m}x\in
			\pi^{-(m+1)}\mathfrak{p}^{m+1}$, so $\pi x\in
			\mathfrak{p}^{m+1}$. This implies that $\phi$ induces
			a surjective map on torsion submodules. 
			
			Now let
			$e_m\geq 0$ be the smallest integer such that
			multiplication with
			$\pi^{e_m}$ kills the torsion submodule of
			$\mathfrak{p}^m/\mathfrak{p}^{m+1}$ or equivalently of
			$\pi^{-m}\mathcal{O}/\mathfrak{p}\otimes_{\mathcal{O}/\mathfrak{p}}
			\mathfrak{p}^m/\mathfrak{p}^{m+1}$. Similarly, let
			$e'_m\geq 0$ be the integer such that multiplication
			with $\pi^{e'_m}$
			kills the torsion submodule of
			$\mathfrak{p}_1^m/\mathfrak{p}_1^{m+1}$. Since $\phi$
			induces a surjective map on torsion submodules, it
			follows that $e_m\geq e'_m$. If $e_m>0$, we claim that
			$e_m>e'_m$. For this it is sufficient to show that
			for every element $x\in \mathfrak{p}^{m}$ such that
			$\pi x \in \mathfrak{p}^{m+1}$, we have $\pi^{-m}\otimes
			x\in \ker(\phi)$. But this is clear: $\phi(\pi^{-m}\otimes
			x)=\pi^{-(m+1)}\pi x\in \mathfrak{p}_1^{m+1}$.

			Repeating the argument for the blow-ups $X_2\rightarrow
			X_1$, $X_3\rightarrow X_2$, and so on, it follows that for fixed $m$, there
			exists a minimal integer $N(m)\geq 0$ such that
			$\mathfrak{p}^m_{n}/\mathfrak{p}^{m+1}_{n}$ is
			torsion free for all $n\geq N(m)$. It remains to see
			that the sequence $N(m)$ is bounded. Consider the
			associated graded ring $\bigoplus_{m\geq 0}
			\mathfrak{p}^m/\mathfrak{p}^{m+1}$. The subset of
			elements which are killed by a power of $\pi$ is an
			ideal of this noetherian ring, hence finitely
			generated. Thus the sequence of numbers $e_m$ is bounded, which
			implies that the sequence $N(m)$ is bounded. This
			completes the proof \ref{gabber:a}.

			Finally, lets prove \ref{gabber:b}. Assume that
			$\mathcal{O}/\mathfrak{p}$ is a discrete valuation
			ring, that $\mathcal{O}_{\mathfrak{p}}$ is regular and
			that $\mathfrak{p}^{m}/\mathfrak{p}^{m+1}$ is a free
			$\mathcal{O}/\mathfrak{p}$-module for every
			${m}\geq 0$. To prove that $\mathcal{O}$ is
			regular, it suffices to show that $\Spec
			\mathcal{O}/\mathfrak{p}\rightarrow \Spec \mathcal{O}$
			is a regular immersion, i.e.~that for every $m\geq 0$
			the natural (surjective) map
			\begin{equation}\label{eq:sym}\Sym^m
				\mathfrak{p}/\mathfrak{p}^2\twoheadrightarrow
				\mathfrak{p}^m/\mathfrak{p}^{m+1}\end{equation}
			is an isomorphism of
			$\mathcal{O}/\mathfrak{p}$-modules. Write $K$ for the
			fraction field of $\mathcal{O}/\mathfrak{p}$. Looking
			at the commutative diagram
			\begin{equation*}
				\begin{tikzcd}
					\mathcal{O}\rar[two
					heads]\dar[hookrightarrow]&
					\mathcal{O}/\mathfrak{p}\dar[hookrightarrow]\\
					\mathcal{O}_{\mathfrak{p}}\rar[two
					heads]&
					K=\mathcal{O}_{\mathfrak{p}}/\mathfrak{p}\mathcal{O}_{\mathfrak{p}}
				\end{tikzcd}
        		 \end{equation*}
			 we see that \eqref{eq:sym} is an isomorphism after
			 tensoring with $K$. In particular, $\Sym^m
			 \mathfrak{p}/\mathfrak{p}^2$ and
			 $\mathfrak{p}^m/\mathfrak{p}^{m+1}$ have the same
			 rank $r$. As
			 $\mathfrak{p}^m/\mathfrak{p}^{m+1}$ is a free
			 $\mathcal{O}/\mathfrak{p}$-module by
			 assumption, it follows that \eqref{eq:sym} can be
			 identified with is a
			 surjective endomorphism  of a free
			 $\mathcal{O}/\mathfrak{p}$-module
			 of rank $r$
			 and hence is an
			 isomorphism.
       %
       %
       %
		\end{proof}
		Using \autoref{lemma:Gabber} we finish the proof of
		\autoref{lemma:pull-back-closed-immersion}.
		Let $\phi:\mathcal{O}_{X',i'(\eta)}\rightarrow
		\mathcal{O}_{\overline{Y}, \eta}$ be the morphism induced
		by $i'$. Then $\phi$ is surjective, as $i'$ is a
		closed immersion by construction, and if $\mathfrak{p}=\ker(\phi)$, then
		$\mathfrak{p}$ is prime. As
		$\mathcal{O}_{\overline{Y},\eta}=\mathcal{O}_{X',i'(\eta)}/\mathfrak{p}$
		is $1$-dimensional, we can apply Gabber's
		\autoref{lemma:Gabber} to
		$\mathcal{O}_{X',i'(\eta)}$ and $\mathfrak{p}$: It shows that after
		replacing $X'$ by a chain of blow-ups with centers  over
		$X'\setminus X$, and $\overline{Y}$ with its proper transform,
		$i'(\eta)$ lies in $X''_{\reg}$. Thus, after
		removing a closed subset of codimension $\geq 2$ from
		$\overline{Y}$ we have $i'(\overline{Y})\subset X'_{\reg}$.

		Moreover, as $i'(\eta)$ is a regular point of
		$i'(\overline{Y})$, there is a regular system of parameters
		$h_0,\ldots, h_{n}$ of
		$\mathcal{O}_{X',i'(\eta)}$ such that
		$(h_1,\ldots, h_n)=\mathfrak{p}=\ker(\mathcal{O}_{X',i'(\eta)}\rightarrow
		\mathcal{O}_{\overline{Y},\eta})$, and such that $h_0$ is the
		uniformizer of the discrete valuation ring
		$\mathcal{O}_{\overline{Y},\eta}=\mathcal{O}_{X',i'(\eta)}/\mathfrak{p}$.

		Without loss of generality we may assume that $X'\setminus X$
		is the support of a Cartier divisor with local equation $g$
		around $i'(\eta)$. Then $g\in \mathfrak{m}_{i'(\eta)}$, where
		$\mathfrak{m}_{i'(\eta)}$ is the maximal ideal of
		$\mathcal{O}_{X',i'(\eta)}$, and we claim that $g$ can be
		written
		\[g=uh_0^m+\sum_{i=1}^n a_i h_i\] with $u\in
		\mathcal{O}_{X',i'(\eta)}^\times$, $a_i \in
		\mathcal{O}_{X',i'(\eta)}$. Indeed, since
		$i'(\overline{Y})\cap X\neq \emptyset$, we see that
		$g$ has nonzero image in
		$\mathcal{O}_{X',i'(\eta)}/(h_1,\ldots,
		h_n)=\mathcal{O}_{\overline{Y},\eta}$ which is a
		discrete valuation ring with uniformizer $h_0$. So
		$g=\bar{u}h_0^m \mod (h_1,\ldots, h_n)$ with
		$\bar{u}\in \mathcal{O}_{\overline{Y},\eta}^\times$.
		Any lift $u$ of $\bar{u}$ to $\mathcal{O}_{X',i'(\eta)}$
		is a unit, so the claim follows. Moreover, $m>0$,
		since $i'(\overline{Y})\not\subset X$.

		If for every $i>0$ the term $a_ih_i$ is divisible by
		$h_0^m$ we are done, because in this case we can write
		$g=h_0^m\cdot\text{unit}$, so around $i'(\eta)$ the
		reduced induced structure on $X'\setminus X$ is
		$V(h_0)$, hence regular, and $i'(\overline{Y})$
		intersects $X'\setminus X$ in $i'(\eta)$
		transversally. 

		If $a_ih_i$ is not divisible by $h_0^m$ for all $i>0$,
		then we blow up $X'$ in $\overline{\{i'(\eta)\}}$ and replace $X'$ by
		this blow-up and $\overline{Y}$ by its proper
		transform (note that this does not change
		$\overline{Y}$, as $i'(\eta)$ is of codimension $1$
		in $i'(\overline{Y})$). Then the local ring $\mathcal{O}_{X',i'(\eta)}$
		has a regular system of parameters $(h_0,h_1/h_0,\ldots,
		h_n/h_0)$. Hence, repeating this process $m$-times, we
		can write
		\[g=h_0^m(u+\sum_{i=1}^na_ih_i/h_0^m)\]
		in $\mathcal{O}_{X',i'(\eta)}$, and we conclude as in
		the previous paragraph.
	\end{proof}
\section{Checking for regular singularities on curves}\label{sec:curves}
We continue to denote by $k$ an algebraically closed field of positive
characteristic $p$, and by $X$ a smooth, connected, separated, finite type
$k$-scheme.


To prove \autoref{thm:main:curves}, we first establish the following easy
lemma:
\begin{Lemma}\label{lemma:poleorderconstructible}
	Let $S$ be a noetherian scheme, $X\rightarrow S$ a smooth morphism of
	finite type $k$-schemes with $X=\Spec
	A$ affine, $U=\Spec A[t^{-1}]$, and $t\in A$ a regular element. Assume that the closed subscheme $D:=V(t)\subset X$ is
	irreducible and smooth over $S$. If
	$g\in A[t^{-1}]$, then the set
	\[ \Pol_{\leq n}(g):=\left\{ s\in S|\; g|_{U_s}\in
	\Gamma(U_s,\mathcal{O}_{U_s})\text{ has pole order } \leq n\text{ along
	} D_s \right\}\]
	is a constructible subset of $S$.
\end{Lemma}
\begin{proof}
	Note that since $D\rightarrow S$ is smooth, $D_s\subset X_s$ is a smooth divisor for
	every $s\in S$, so it makes sense to talk about the pole order of $g|_{U_s}$ along
	$D_s$.

	Since $\Pol_{\leq n}(g)=\Pol_{\leq 0}(t^ng)$, it suffices
	to show that $\Pol_{\leq 0}(g)$ is constructible.

	The element $g$ defines a commutative diagram of $S$-schemes
	\begin{equation*}
		\begin{tikzcd}
			U\rar[hookrightarrow]\dar{g}& X\dar{g}\\
			\A^1_S\rar[hookrightarrow]&\P^1_S.
		\end{tikzcd}
	\end{equation*}
	 The image $g(X)\subset\P_S^1$ is a constructible set, so
	 $g(X)\cap(\{\infty\}\times S)$ is a constructible subset of $\P^1_S$. If
	 $\pr:\P^1_S\rightarrow S$ is the structure morphism of $\P^1_S$, then
	 $\pr(g(X)\cap (\{\infty\}\times S))$ is a constructible subset of $S$.
	 Finally note that $S\setminus \Pol_{\leq 0}(g)=\pr( (g(X)\cap (\{\infty\}\times
	 S))$.
\end{proof}
We are now ready to prove \autoref{thm:main:curves} with respect to a fixed
good partial compactification.
\begin{Proposition}\label{prop:curveCriterionPC}
	Let $(X,\overline{X})$ be a good partial compactification and $E$ a
	stratified bundle on $X$. Assume that for every $k$-morphism
	$\bar{\phi}:\overline{C}\rightarrow \overline{X}$ with
	$\overline{C}$ a smooth $k$-curve, the stratified bundle $\phi^*E$ is
	$(C, \overline{C})$-regular singular, where $C:=\bar{\phi}^{-1}(X)$,
	$\phi:=\bar{\phi}|_C$. 
	
	If the base field $k$ is uncountable, then
	$E$ is $(X,\overline{X})$-regular singular. 
\end{Proposition}
\begin{proof}
	Let $k$ be an uncountable algebraically closed field of characteristic $p>0$.  We immediately reduce to the case
	where $X$ is connected and $\dim X \geq 2$. By removing a closed set
	of codimension $\geq 2$ from $\overline{X}$ we may assume that
	$D:=(\overline{X}\setminus X)_{\red}$ is a smooth divisor; by treating
	its components separatedly, we may assume that $D$ is irreducible with
	generic point $\eta$.
	Shrinking $\overline{X}$ further, we may assume that \begin{itemize}
	\item $\overline{X}=\Spec A$ is affine, \item there exist \'etale coordinates
		$x_1,\ldots, x_n\in A$ such that $D=V(x_1)$.  \item $E$
			corresponds to a free $A[x_1^{-1}]$-module, say with
			basis $e_1,\ldots, e_r$.  \end{itemize} If we write
			$\delta^{(m)}_{x_1}:=x_1^m\partial_{x_1}^{(m)}\in
			\mathscr{D}_{X/k}$ (see \autoref{rem:logDiffOps}), then for $f\in A$ we can
			also write 
			\[\delta_{x_1}^{(m)}(fe_i)=\sum_{j=1}^r
				b_{ij}^{(m)}(f) e_j,\text{ with }
				b_{ij}^{(m)}(f)\in A[x_1^{-1}].\] 
		To show that
		$E$ is regular singular, it suffices to show
		that the pole order of the elements $b_{ij}^{(m)}(f)$ along
		$x_1$ is bounded by some $N\in \N$, because then the
		$\mathscr{D}_{\overline{X}/k}(\log D)$-module
		generated by $\sum_{i=1}^r e_i A$ is contained
		in $\left(\bigoplus_{i=1}^r x_1^{-N}e_i A\right)$, and thus
		finitely generated over $A$.

		Write $A$ as the
		quotient of a polynomial ring $k[y_1,\ldots,
		y_d]$ and $\bar{y}_i$ for the image of $y_i$ in $A$. It then suffices to show that the pole order of
		$b_{ij}^{(m)}(\bar{y}_c^{h})$ has a common upper bound, for
		$1\leq i,j\leq r$, $1\leq c \leq d$, $m, h\geq 0$.

		Define $S:=\A^{n-1}_k=\Spec k[x_2,\ldots, x_n]$.  We get a
		commutative diagram 
		\begin{equation}\label{eqn:fibration}
			\begin{tikzcd}
				X\rar[hookrightarrow]&\overline{X}\rar{\text{\'etale}}\ar[swap]{dr}{\text{smooth}}&\A^1_S\dar\\
			&&S
			\end{tikzcd}
		\end{equation}
		We are now in the situation of
		\autoref{lemma:poleorderconstructible}. For $N\in \N$ consider the constructible
			sets $\Pol_{\leq
			N}(b_{ij}^{(m)}(\bar{y}_{c}^h))\subset S$, and define
			\[\mathbf{P}_{\leq N}:=\bigcap_{i,j,m,c,h}
				\overline{\Pol_{\leq
				N}(b^{(m)}_{ij}(\bar{y}_c^h))}.\] This is a
				closed subset of $S$.  Now since for every
				closed point $s\in S$ the fiber $X_s$
				is a regular curve over $k$ meeting $D$
				transversally, we see that by
				assumption $E|_{X_s}$ is
				$(X_s,\overline{X}_s)$-regular singular.  But
				this means that there is some $N_s\geq 0$,
				such that $s\in \mathbf{P}_{\leq N_s}$.  In
				other words, the union $\bigcup_{N\geq 0}
 				\mathbf{P}_{\leq N}$ contains all closed
				points of $S$.
				Since $k$ is uncountable and since the
				$\mathbf{P}_{\leq N}$ are closed subsets of
				$S$, this means there exists some $N_0\geq 0$,
				such that $\mathbf{P}_{\leq N_0}=S$.
				The definition
				of $\mathbf{P}_{\leq N_0}$ and
				\autoref{lemma:poleorderconstructible} imply
				that the sets $\Pol_{\leq
				N_0}(b_{ij}^{(m)}(\bar{y}_c^h))$ are
				\emph{dense constructible} subsets of $S$.
				But a dense constructible subset of an
				irreducible noetherian space contains an open
				dense subset by \cite[Prop.~10.14]{Goertz}.
				This shows that the pole order of
				$b_{ij}^{(m)}(\bar{y}_c^{h})$ along $x_1$ is
				bounded by $N_0$, and thus that $E$ is
				$(X,\overline{X})$-regular singular.
\end{proof}

Now we can easily finish the proof of \autoref{thm:main:curves}.
\begin{proof}[Proof of \autoref{thm:main:curves}]
	We have proved \autoref{thm:main:pullback} which shows that if $E$ is
	a regular singular stratified bundle on $X$, then $\phi^*E$ is regular
	singular for every map $\phi:C\rightarrow X$ with $C$ a smooth
	$k$-variety.

	For the converse, assume that $\phi^*E$ is regular singular for every
	morphism $\phi:C\rightarrow X$ with $C$ a smooth $k$-curve.
	We have to show that $E$ is $(X,\overline{X})$-regular singular with
	respect to every good partial compactification
	$(X,\overline{X})$. But the assumptions of
	\autoref{prop:curveCriterionPC} are satisfied, so $E$ is
	$(X,\overline{X})$-regular singular.
\end{proof}

Finally, we give a proof of \autoref{cor:arbitraryField}.

\begin{proof}[Proof of \autoref{cor:arbitraryField}]
	We need to show that $E$ is $(X,\overline{X})$-regular singular for
	every
	good partial compactification $(X,\overline{X})$, whenever the
	condition of this corollary is satisfied.
	We may assume that $\overline{X}\setminus X$ is irreducible.	As in
	the proof of \autoref{thm:main:curves}, we reduce
	to $\overline{X}$ affine and $E$ free, so that showing that $E$ is
	$(X,\overline{X})$-regular singular boils down to showing
	that the pole order of a certain set of functions in $k(X)$ has a
	common bound. This is independent of the coefficients, so we may base change to a
	field $K\supset k$, $K$ algebraically closed and uncountable (e.g.
	$K:=\overline{k\llparen t\rrparen}$). Then we can apply the theorem, to see that
	$E_K$ is regular singular if it is regular singular along all smooth $K$-curves.
	But every such curve is defined over a subextension $k'$ of	$K/k$, finitely generated over $k$.  The corollary follows.
\end{proof}


\end{document}